\documentclass{article}
\usepackage{amsfonts}
\usepackage{amsmath}

\newtheorem{theorem}{Theorem}

\newtheorem{lemma}[theorem]{Lemma}

\newtheorem{proposition}[theorem]{Proposition}
\newtheorem{remark}[theorem]{Remark}

\newenvironment{proof}[1][Proof]{\noindent\textbf{#1.} }{\ \rule{0.5em}{0.5em}}
\input{tcilatex}
\begin{document}

\title{A Linear Division-Based Recursion with Number Theoretic Applications}
\author{Jonathan L. Merzel}
\maketitle

\begin{abstract}
A simple remark on infinite series is presented. \ This applies to a
particular recursion scenario, which in turn has applications related to a
classical theorem on Euler's phi-function and to recent work by Ron Brown on
natural density of square-free numbers.
\end{abstract}

\section{A Basic Fact about Infinite Series}

In a recent paper \cite{Brown}, Ron Brown has computed the natural density
of the set of square-free numbers divisible by $a$ but relatively prime to $%
b $, where $a$ and $b$ are relatively prime square-free integers. \ Here we
note a simple remark on infinite series, one of whose consequences
generalizes a key argument in that work. \ We then derive a consequence of a
well-known result on the Euler $\varphi $-function. \ The "$m=p$" case of
that consequence follows from En-Naoui\cite{En} who anticipates some of our
arguments..

\bigskip

\begin{remark}
Let $\underset{i=1}{\overset{\infty }{\sum }}a_{i}$ be an absolutely
convergent series of complex numbers, and for $i\geq 1$, $f_{i}:%
\mathbb{N}
\cup \{0\}\rightarrow 
\mathbb{C}
$ with $\underset{N\rightarrow \infty }{\lim }f_{i}(N)=D$ (independent of $i$%
) and the $f_{i}$ uniformly bounded. \ Then $\underset{N\rightarrow \infty }{%
\lim }\underset{i=1}{\overset{\infty }{\sum }}a_{i}f_{i}(N)=D\underset{i=1}{%
\overset{\infty }{\dsum }}a_{i}.$
\end{remark}

\begin{proof}
This is a special case of the Lebesgue Dominated Convergence Theorem (using
the counting measure and applied to the sequence $\{a_{i}f_{i}(n)\}_{n=1}^{%
\infty }$). \ To preserve the elementary character of the arguments here, we
give an "Introductory Analysis" proof. \ \newline
Let $\varepsilon >0$ be given. $\ $By uniform boundedness, there is a
constant $B$ for which $\left\vert f_{i}(N)-D\right\vert <B$ for all $i$ and 
$N$. \ Choose $k\in 
\mathbb{N}
$ with$\underset{i=k+1}{\overset{\infty }{\sum }}\left\vert a_{i}\right\vert
<\frac{\varepsilon }{2B}$, and choose M such that for all $N\geq M$ and $%
1\leq i\leq k,~$ $\left\vert f_{i}(N)-D\right\vert <\varepsilon /(1+2%
\underset{j=1}{\overset{k}{\sum }}\left\vert a_{j}\right\vert ).$ $\ $Then
we have for $N\geq M$ 
\begin{eqnarray*}
\left\vert \underset{i=1}{\overset{\infty }{\sum }}a_{i}f_{i}(N)-D\underset{%
i=1}{\overset{\infty }{\dsum }}a_{i}\right\vert &=&\left\vert \underset{i=1}{%
\overset{\infty }{\sum }}a_{i}(f_{i}(N)-D)\right\vert \\
&\leq &\underset{i=1}{\overset{k}{~\sum }}\left\vert a_{i}\right\vert
\left\vert (f_{i}(N)-D)\right\vert +\underset{i=k+1}{\overset{\infty }{~\sum 
}}\left\vert a_{i}\right\vert \left\vert (f_{i}(N)-D)\right\vert \\
&<&\underset{i=1}{\overset{k}{~\sum }}\left\vert a_{i}\right\vert \cdot
\varepsilon /(1+2\underset{i=1}{\overset{k}{\sum }}\left\vert
a_{i}\right\vert )+\frac{\varepsilon }{2B}\cdot B<\varepsilon
\end{eqnarray*}
\end{proof}

\section{A Consequence and Some Applications}

For all applications of the remark above, we first derive the following
consequence involving a "linear division-based"\ recursion.

\begin{lemma}
Let, $F,G:%
\mathbb{N}
\cup \{0\}\rightarrow 
\mathbb{C}
$, $1<m\in 
\mathbb{N}
,~\alpha ,\beta ,D\in 
\mathbb{C}
$ satisfy the conditions (1) $\underset{N\rightarrow \infty }{\lim }F(N)/N=D$%
, (2) $\left\vert \beta \right\vert <m$, (3) $G(N)=\alpha F(\left\lfloor
N/m\right\rfloor )+\beta G(\left\lfloor N/m\right\rfloor )$, and (4) $%
F(0)=G(0)=0$. Then $\underset{N\rightarrow \infty }{\lim }G(N)/N=\frac{%
D\alpha }{m-\beta }.$
\end{lemma}

\begin{proof}
Recursively expand (using condition (3) and $\left\lfloor \left\lfloor
a/b\right\rfloor /c\right\rfloor =\left\lfloor a/(bc)\right\rfloor $ for
positive integers $a,b,c$ ) we have for $N>0$ 
\begin{equation}
G(N)/N=\frac{\alpha }{m}\cdot \frac{F(\left\lfloor N/m\right\rfloor )}{N/m}+%
\frac{\alpha \beta }{m^{2}}\frac{F(\left\lfloor N/m^{2}\right\rfloor )}{%
N/m^{2}}+\cdots +\frac{\alpha \beta ^{j-1}}{m^{j}}\frac{F(\left\lfloor
N/m^{j}\right\rfloor )}{N/m^{j}}+\frac{\alpha \beta ^{j-1}}{m^{j}}\frac{%
G(\left\lfloor N/m^{j}\right\rfloor )}{N/m^{j}}  \tag{*}
\end{equation}%
$.$ \ \ \ By properties (1), (2) and (4), this implies we have 
\begin{equation*}
G(N)/N=\dsum\limits_{i=1}^{\infty }\frac{\alpha \beta ^{i-1}}{m^{i}}\frac{%
F(\left\lfloor N/m^{i}\right\rfloor )}{N/m^{i}}
\end{equation*}

After all, for any fixed N this is actually a finite sum by (4) and the
final term in display (*) above is 0 for large $j$. Now by Lemma 1, taking $%
a_{i}=\frac{\alpha \beta ^{i-1}}{m^{i}}$ and $f_{i}(N)=\frac{F(\left\lfloor
N/m^{i}\right\rfloor )}{N/m^{i}}$, it follows that $\underset{N\rightarrow
\infty }{\lim }G(N)/N=D\dsum\limits_{i=1}^{\infty }\frac{\alpha \beta ^{i-1}%
}{m^{i}}=\frac{D\alpha }{m-\beta }$.
\end{proof}

We can derive some simple applications.\newline

Application 1. \ Let $m$ be an integer greater than 1. \ Call an integer $n$
oddly divisible by $m$ if the largest nonnegative integer $t$ with $m^{t}|n$
is odd. \ Similarly define evenly divisible. (Note that by this definition,
a number not divisible by $m$ is evenly divisible by $m$.) \ Set $F(n)=n$
and $G(n)=$ $\left\vert \left\{ i\in 
\mathbb{N}
:1\leq i\leq n\text{, }i\text{ oddly divisible by }m\right\} \right\vert $.
Since there is a 1-1 correspondence between $\left\{ i\in 
\mathbb{N}
:1\leq i\leq n\text{, }i\text{ oddly divisible by }m\right\} $ and $\left\{
i\in 
\mathbb{N}
:1\leq i\leq \left\lfloor n/m\right\rfloor \text{ and }i\text{ is evenly
divisible by }m\right\} $, we quickly see that $G(n)=F(\left\lfloor
n/m\right\rfloor )-G(\left\lfloor n/m\right\rfloor )$. \ Now apply the Lemma
with $D=\alpha =-\beta =1$ to get $\underset{N\rightarrow \infty }{\lim }%
G(N)/N=\frac{1}{m+1}.$ \ So the natural density of numbers oddly divisible
by $m$ is $\frac{1}{m+1}$. \ (This is also easily arrived at by an
inclusion-exclusion argument.)\newline

Application 2. In Brown\cite{Brown} \ the natural density of the set of
square-free numbers divisible by primes $p_{1},\cdots ,p_{k}$ is shown to be 
$6/\pi ^{2}\tprod\limits_{i=1}^{k}\frac{1}{p_{k}+1}$. \ (In fact, he more
generally computes the density of the set of such numbers also not divisible
by a further set of primes and reduces that problem to this one.) \ Using
that the natural density of the set of square-free numbers is $6/\pi ^{2}$,
the cited result follows directly from [?] Lemma 3, which states that, for a
square-free integer $t$ and a prime $p$ not dividing $t$, if the natural
density of the set of square-free numbers divisible by $t$ is $D$, then the
natural density of the set of square-free numbers divisible by $tp$ is $%
D/(p+1)$. \ To do this (converting to our notation), letting $C$ be the set
of square-free numbers, $F(x)$ $=$ $\left\vert \left\{ r\in C:t|r,r\leq
x\right\} \right\vert $ and $G(x)=\left\vert \left\{ r\in C:pt|r,r\leq
x\right\} \right\vert $ Brown quickly establishes that $F(x/p)=G(x/p)+G(x).$
\ Noting that we can replace arguments here with their greatest integers,
and that all hypotheses are in place, we can apply Lemma 2 with $\alpha
=1,~\beta =-1,~m=p$ to arrive at $\underset{N\rightarrow \infty }{\lim }%
G(N)/N=\frac{D}{p+1}.$

\section{Application to a Classical Theorem on Euler's $\protect\varphi $%
-function}

It is well-known that $\underset{N\rightarrow \infty }{\lim }\left(
\dsum\limits_{n=1}^{N}\frac{\varphi (n)}{n}\right) /N=$ $6/\pi ^{2}$.\ \
(See for example \cite{Erd}.)

\ From this we can derive the following proposition, where we sum only over
multiples of an integer $m$:

\begin{proposition}
\bigskip Let $m$ be a positive integer, and let $p_{1},\cdots ,p_{k}$ the
distinct prime divisors of $m$. \ Then 
\begin{equation*}
\underset{N\rightarrow \infty }{\lim }\left( \dsum\limits_{m|n\leq N}\frac{%
\varphi (n)}{n}\right) /N=\frac{6}{\pi ^{2}m}\prod\limits_{j=1}^{k}\frac{%
p_{j}}{1+p_{j}}
\end{equation*}
\end{proposition}

Some numerical evidence:

$N=1000,~m=5.$ \ Here \bigskip $\frac{\tsum\limits_{5|n\leq 1000}\frac{%
\varphi (n)}{n}}{1000}\approx .1016$ \ while $\frac{6}{5\pi ^{2}}\cdot \frac{%
5}{6}\approx .1013$.

$N=100000,~m=200.$ \ Here \bigskip $\frac{\tsum\limits_{200|n\leq 100000}%
\frac{\varphi (n)}{n}}{100000}\approx .001691$, while $\frac{6}{200\pi ^{2}}%
\cdot \frac{2}{3}\cdot \frac{5}{6}\approx .001689$.

$N=10000000,~m=12348.$ \ Here \bigskip $\frac{\tsum\limits_{12348|n\leq
1000000}\frac{\varphi (n)}{n}}{1000000}\approx .00002153$, while $\frac{6}{%
12348\pi ^{2}}\cdot \frac{2}{3}\cdot \frac{3}{4}\cdot \frac{7}{8}\approx
.00002154$.\newline

\begin{proof}
The result will follow inductively from the following \ \newline
Claim : \ Let $p$ be a prime, $k$ a positive integer and $t$ an positive
integer not divisible by $p$. \ Then if $\underset{N\rightarrow \infty }{%
\lim }\left( \dsum\limits_{t|n\leq N}\frac{\varphi (n)}{n}\right) /N=L$, it
follows that \newline
$\underset{N\rightarrow \infty }{\lim }\left( \dsum\limits_{tp^{j}|n\leq N}%
\frac{\varphi (n)}{n}\right) /N=\frac{L}{p^{j-1}(p+1)}.$\newline
To establish the claim, we first handle the case $j=1$. \ We set $%
F(N)=\dsum\limits_{t|n\leq N}\frac{\varphi (n)}{n},~G(N)=\dsum\limits_{pt|n%
\leq N}\frac{\varphi (n)}{n}$. \ We can bijectively correspond the set $A$
of integers divisible by $t$ and less than or equal to $N/p$ with the set $B$
of multiples of $pt$ less than or equal to $N$ by multiplication by $p$. \
We write $A=A_{1}\cup A_{2\text{,}}$, with multiples of $p$ in $A_{1}$ and
nonmultiples of $p$ in $A_{2}$, and note that (from the usual computation of 
$\varphi $ in terms of prime factorization) for $n\in A_{1},\varphi (n)/n=$ $%
\varphi (pn)/(pn)$, while for $n\in A_{2},\varphi (n)/n=\frac{p}{p-1}$ $%
\varphi (pn)/(pn)$. \ So%
\begin{eqnarray*}
G(N) &=&\dsum\limits_{pt|n\leq N}\frac{\varphi (n)}{n}=\dsum\limits_{n\in
A_{1}}\frac{\varphi (np)}{np}+\dsum\limits_{n\in A_{2}}\frac{\varphi (np)}{np%
}=\dsum\limits_{n\in A_{1}}\frac{\varphi (n)}{n}+\frac{p-1}{p}%
\dsum\limits_{n\in A_{2}}\frac{\varphi (n)}{n} \\
&=&\frac{p-1}{p}F(\left\lfloor N/p\right\rfloor )+\frac{1}{p}G(\left\lfloor
N/p\right\rfloor
\end{eqnarray*}%
Applying our lemma with $m=p$, $\alpha =\frac{p-1}{p}$, $\beta =\frac{1}{p}$%
, $D=L$ we get 
\begin{equation*}
\underset{N\rightarrow \infty }{\lim }G(N)/N=\frac{D\alpha }{m-\beta }=\frac{%
L}{p+1}.
\end{equation*}%
Now we can proceed to the general case of the claim. \ We now bijectively
correspond the set $A$ of integers divisible by $t$ and less than or equal
to $N/p^{j}$ with the set $B$ of multiples of $p^{j}t$ less than or equal to 
$N$ by multiplication by $p^{k}$, and similarly $j=1$ case write $%
A=A_{1}\cup A_{2\text{,}}$, with multiples of $p$ in $A_{1}$ and
nonmultiples of $p$ in $A_{2},$ . \ Then%
\begin{eqnarray*}
\dsum\limits_{p^{j}t|n\leq N}\frac{\varphi (n)}{n} &=&\dsum\limits_{t|i\leq
N/p^{j}}\frac{\varphi (p^{j}i)}{p^{j}i}=\dsum\limits_{n\in A_{1}}\frac{%
\varphi (p^{j}i)}{p^{j}i}+\dsum\limits_{n\in A_{2}}\frac{\varphi (p^{j}i)}{%
p^{j}i} \\
&=&\dsum\limits_{n\in A_{1}}\frac{p^{j-1}(p-1)\varphi (i)}{p^{j}i}%
+\dsum\limits_{n\in A_{2}}\frac{\varphi (p^{j}i)}{p^{j}i} \\
&=&\frac{p-1}{p}\dsum\limits_{n\in A_{1}}\frac{\varphi (i)}{i}%
+\dsum\limits_{n\in A_{2}}\frac{\varphi (i)}{i} \\
&=&\frac{p-1}{p}\dsum\limits_{t|n\leq N/p^{j}}\frac{\varphi (i)}{i}+\frac{1}{%
p}\dsum\limits_{pt|i\leq N/p^{j}}\frac{\varphi (i)}{i}
\end{eqnarray*}%
\newline
Dividing through by $N$ we get%
\begin{eqnarray*}
\dsum\limits_{p^{j}t|n\leq N}\frac{\varphi (n)}{n}/N &=&\frac{p-1}{p^{j+1}}%
\frac{\dsum\limits_{t|n\leq N/p^{j}}\frac{\varphi (i)}{i}}{N/p^{j}}+\frac{1}{%
p^{j+1}}\frac{\dsum\limits_{pt|i\leq N/p^{j}}\frac{\varphi (i)}{i}}{N/p^{j}}
\\
&\rightarrow &\frac{L(p-1)}{p^{j+1}}+\frac{L}{p^{j+1}(p+1)}=\frac{L}{%
p^{j-1}(p+1)}\text{ as }N\rightarrow \infty
\end{eqnarray*}%
where the first limit of the first term is given by the\ hypothesis $%
\underset{N\rightarrow \infty }{\lim }\left( \dsum\limits_{t|n\leq N}\frac{%
\varphi (n)}{n}\right) /N=L$ and the limit of the second term follows from
the $j=1$ case above. \ That concludes the proof of the claim, and hence the
proposition.
\end{proof}

\end{document}